\documentclass{amsart}

\usepackage{amssymb}
\usepackage{amsthm}

 \newtheorem{theorem}{Theorem}[section]
  
 \newtheorem{lemma}[theorem]{Lemma}
 \newtheorem{proposition}[theorem]{Proposition}

\newcommand{\la}{\langle}
\newcommand{\ra}{\rangle}

\begin{document}


\title[Strongly Closed $2$-subgroups]{\bf Influence of strongly closed $2$-subgroups on the structure of finite groups }

\author{ Hung P. Tong-Viet}

\address{Department of Mathematics, University of Birmingham, Edgbaston, Birmingham, B15 2TT, UK}

\email{tongviet@maths.bham.ac.uk}




\date{\today}

\begin{abstract}
Let $H\leq K$ be subgroups of a group $G.$ We say that $H$ is
strongly closed in $K$ with respect to $G$ if whenever $a^g\in K,$
where $a\in H, g\in G,$ then $a^g\in H.$ In this paper, we
investigate the structure of a group G under the assumption that
every subgroup of order $2^m$ (and $4$ if $m=1$) of a  $2$-Sylow
subgroup $S$ of G is strongly closed in $S$ with respect to $G.$
Some results related to $2$-nilpotence and supersolvability of a
group $G$ are obtained. This is a complement to Guo and Wei (J.
Group Theory \textbf{13} (2010), no. 2, 267--276).
\end{abstract}

\thanks{Support from the Leverhulme Trust is acknowledged}
\maketitle

\section{Introduction}
All groups are finite. Let $H\leq K$ be subgroups of a group $G.$ We
say that $H$ is \emph{strongly closed} in $K$ with respect to $G$ if
whenever $a\in H, a^g\in K,$ where $g\in G$ then $a^g \in H.$ We
also say that $H$ is strongly closed in $G$ if $H$ is strongly
closed in $N_G(H)$ with respect to $G.$ The structure of groups
which possess a strongly closed $p$-subgroup has been extensively
studied. One of the most interesting results is due to Goldschmidt
\cite{G74} which classified groups with an abelian strongly closed
$2$-subgroup. This result is a generalization of the celebrated
Glauberman $Z^*$-theorem. These results play an important role in
the proof of the classification of the finite simple groups.
Recently, Bianchi et al. in \cite{Bia}, called a subgroup $H,$ an
\emph{$\mathcal{H}$-subgroup} of $G$ if $H^g\cap N_G(H)\leq H$ for
all $g\in G.$ It is easy to see that these two definitions coincide.
With this concept, they gave a new characterization of supersolvable
groups in which normality is a transitive relation which are called
supersolvable $\mathcal{T}$-groups. In more detail, it is shown that
every subgroup of $G$ is strongly closed in $G$ if and only if $G$
is a supersolvable $\mathcal{T}$-group (see \cite[Theorem
$10$]{Bia}). Some local versions of this result have been studied in
\cite{Asa1} and \cite{Guo}. For example, Asaad (\cite[Theorem
$1.1$]{Asa1}) proved that $G$ is $p$-nilpotent if and only if every
maximal subgroup of a $p$-Sylow subgroup $P$ of $G$ is strongly
closed in $G$ and $N_G(P)$ is $p$-nilpotent. Guo and Wei
(\cite[Theorem $3.1$]{Guo}) showed that whenever $p$ is odd and $P$
is a $p$-Sylow subgroup of $G,$ $G$ is $p$-nilpotent if and only if
$N_G(P)$ is $p$-nilpotent and either $P$ is cyclic or every
nontrivial proper subgroup of a given order of $P$  is strongly
closed in $G.$  Also these results still hold without the
$p$-nilpotence assumption on $N_G(P)$ if $p$ is the smallest prime
divisor of the order of $G.$ The purpose of this paper is to prove
the following theorem, which is a complement to \cite[Theorem
$3.1$]{Guo}.

\begin{theorem}\label{th1}
Let $P\in Syl_2(G)$ and $D\leq P$ with $1<|D|<|P|.$ If $P$ is either
cyclic or every subgroup of $P$ of order $|D|$ \emph{(}and $4$ if
$|D|=2$\emph{)} is strongly closed in $G,$ then $G$ is
$2$-nilpotent.
\end{theorem}
The following example shows that the additional assumption when $|D|=2$ in Theorem \ref{th1} is necessary.

\noindent {\bf Example.} Let $G=SL_2(17).$ Then if $P\in Syl_2(G)$
then $P\cong Q_{32},$ a quaternion group of order $32.$ Moreover
$P$ is maximal in $G$ and hence $N_G(P)=P$ is $2$-nilpotent in $G.$
Clearly, the center of $G$ is a unique subgroup of order $2$ and so
it is strongly closed in $G.$ However $G$ is not $2$-nilpotent.

Theorem \ref{th1} above and \cite[Theorem $3.4$]{Guo} now yield:

\begin{theorem}\label{th2}
Let $p$ be the smallest prime divisor of $|G|$ and $P\in Syl_p(G).$
If $P$ is cyclic or $P$ has a subgroup $D$ with $1<|D|<|P|$ such
that every subgroup of $P$ of order $|D|$ \emph{(}and $4$ if
$|D|=2$\emph{)} is strongly closed in $G,$ then $G$ is
$p$-nilpotent.
\end{theorem}

We can now drop the odd order assumption on Theorems $3.5$ and $3.6$ in \cite{Guo}.

\begin{theorem}\label{th3}
If every non-cyclic Sylow subgroup $P$ of $G$ has a subgroup $D$
with $1<|D|<|P|$ such that every subgroup of $P$ of order $|D|$
\emph{(}and $4$ if $|D|=2$\emph{)} is strongly closed in $G,$ then
$G$ is supersolvable.
\end{theorem}

\begin{theorem}\label{th4}
Let $E$ be a normal subgroup of $G$ such that $G/E$ is
supersolvable. If every non-cyclic Sylow subgroup $P$ of $E$ has a
subgroup $D$ with $1<|D|<|P|$ such that every subgroup of $P$ of
order $|D|$ \emph{(}and $4$ if $|D|=2$\emph{)} is strongly closed in
$G,$ then $G$ is supersolvable.
\end{theorem}

\section{Preliminaries}
In this section, we collect some results needed in the proofs of the main theorems.
\begin{lemma}\emph{ (Schur-Zassenhauss \cite[Theorem $6.2.1$]{gor}).}\label{SZ}
If $P$ is a normal $2$-Sylow subgroup of $G$ then $G$ possesses a
complement Hall-$2'$-subgroup.
\end{lemma}

\begin{lemma}\label{cyclic}\emph{(\cite[Theorem $7.6.1$]{gor}).}
If a $2$-Sylow subgroup of $G$ is cyclic then $G$ is $2$-nilpotent.
\end{lemma}
\begin{lemma}\label{max}\emph{(\cite[Corollary $1.2$]{Asa1}).}
Let $P$ be a $2$-Sylow subgroup of $G.$ Then $G$ is $2$-nilpotent if
and only if every maximal subgroup of $P$ is strongly closed in $G.$
\end{lemma}

\begin{lemma}\label{quot} Suppose that $H$ is a strongly closed $p$-subgroup of $G.$

$(a)$ If $H\leq L\leq G$ then $H$ is strongly closed in $L;$

$(b)$ If $\bar{G}$ is a homomorphic image of $G,$ then $\bar{H}$  is strongly closed in $\bar{G}$ and $N_{\bar{G}}(\bar{H})=\overline{N_G(H)};$

$(c)$ If $H$ is subnormal in $G$ then $H\unlhd G.$
\end{lemma}
\begin{proof} $(a)$ is \cite[Lemma $7(2)$]{Bia} and $(c)$ is \cite[Theorem $6(2)$]{Bia}. Finally $(b)$ is \cite[$(2.2)(a)$]{Gold}.
\end{proof}

\begin{lemma}\label{Syl}\emph{(\cite[Corollary B$3$]{Gold}).} Suppose that $H$ is a strongly closed $2$-subgroup of $G$ and $N_G(H)/C_G(H)$
is a $2$-group. Then $H\in Syl_2(\la H^G\ra).$
\end{lemma}

\begin{lemma}\label{min}\emph{(\cite[Satz $4.5.5$]{Hup}).} If every element of order $2$ and $4$ of $G$ are central then $G$ is $2$-nilpotent.
\end{lemma}

\begin{lemma}\label{maxnil}\emph{(\cite[Baumann]{Bau}).}
If $G$ is a non-abelian simple group in which a $2$-Sylow subgroup
of $G$ is maximal, then $G$ is isomorphic to $L_2(q),$ where $q$ is
a prime number of the form $2^m\pm 1\geq 17.$
\end{lemma}

A \emph{component} of $G$ is a subnormal quasisimple subgroup of
$G.$ Denote by $E(G)$ the subgroup of $G$ generated by all
components of $G.$ Then the \emph{generalized Fitting subgroup}
$F^*(G)$  of $G$ is a central product of $E(G)$ and the Fitting
subgroup $F(G)$ of $G.$
\begin{lemma}\label{genfit}\emph{(\cite[Theorem $9.8$]{Isa}).} $C_G(F^*(G))\leq F^*(G).$
\end{lemma}

\begin{lemma}\label{four}\emph{(\cite[Problem $4$D$.4,$ p. $146$]{Isa}).}
Let $A$ act via automorphisms on a $2$-group $P,$ where $|A|$ is
odd. If $A$ centralizes every element of order $2$ and $4$ in $P,$
then $A$ acts trivially on $P.$
\end{lemma}

The following result is a special case of \cite[Lemma $2.10$]{Guo}.
\begin{lemma}\label{minnorm}
Let $P$ be an elementary abelian $2$-subgroup of $G$ and $D$ a
subgroup of $P$ with $1<|D|<|P|.$ If every subgroup of $P$ of order
$|D|$ is normal in $G,$ then every minimal subgroup of $P$ is
central in $G.$
\end{lemma}
\begin{proof}
It follows from \cite[Lemma $2.10$]{Guo} that every minimal subgroup
of $P$ is normal in $G.$ As minimal subgroups of $P$ are cyclic of
order $2,$ they are all central.
\end{proof}

\begin{lemma}\label{fixact}
Let $A$ be an odd order group acting on a $2$-group $P.$ Let $D\leq
P$ with $1<|D|<|P|.$ If every subgroup of $P$ of order $|D|$
\emph{(}and $4$ if $|D|=2$\emph{)} is $A$-invariant, then $A$ acts
trivially on $P.$
\end{lemma}
\begin{proof}
We can assume that $|D|\geq 4.$ Let $\mathcal{D}=\{E\leq
P\::\:|E|=|D|\}.$ Suppose that $\la \mathcal{D}\ra<P.$ If $|\la
\mathcal{D}\ra|>|D|,$ then by inductive hypothesis, $A$ centralizes
$\la \mathcal{D}\ra,$ so that it centralizes every subgroup of $P$
of order $2$ and $4,$ hence the result follows from Lemma
\ref{four}. If $|\mathcal{D}|=|D|,$ then $P$ has a unique subgroup
of order $|D|.$ As $2<|D|<|P|,$ $P$ must be cyclic and thus $A$
centralizes $P$ by applying Lemma \ref{cyclic} to the semidirect
product $A\ltimes P.$  Therefore, we can assume that $\la
\mathcal{D}\ra=P.$ Next, if $A$ centralizes every element of
$\mathcal{D},$ then as $|D|\geq 4,$ $A$ centralizes every element of
order $2$ and $4,$ and we are done by using Lemma \ref{four}. Hence
there exists $E\in \mathcal{D}$ such that $[E,A]\neq 1.$ It follows
that $\Phi(P)\leq E,$ otherwise, $E<E\Phi(P)<P,$ and by applying the
inductive hypothesis for $E\Phi(P),$ $A$ would centralize $E,$ which
contradicts the choice of $E,$ thus prove the claim. If $\Phi(P)$ is
trivial, then $P$ is elementary abelian, and hence the result
follows from Lemma \ref{minnorm}. Thus $\Phi(P)>1.$ Assume that
$|E/\Phi(P)|\geq 2.$ By Lemma \ref{minnorm} again, $A$ centralizes
$P/\Phi(P),$ and then $[P,A]\leq \Phi(P).$ By Coprime Action
Theorem, $A$ acts trivially on $P$ and we are done. Thus we assume
that $E=\Phi(P).$ For any $F\in \mathcal{D}-\{\Phi(P)\},$ we have
$|F|=|\Phi(P)|$ and $\Phi(P)\neq F,$ it follows that $F<F\Phi(P)<P$
and $F\Phi(P)$ is $A$-invariant. By inductive hypothesis, $A$
centralizes $F,$ and hence $P,$ as $P$ is generated by
$\mathcal{D}-\{\Phi(P)\}.$ The proof is now complete.

\end{proof}
\section{Proofs of the main results}
\begin{proposition}\label{fixstrong}
Let $P\in Syl_2(G)$ and $D\leq P$ with $2<|D|<|P|.$ Assume that
either $P$ is cyclic or every subgroup of $P$ of order $|D|$ is
strongly closed  in $G,$ then $G$ is $2$-nilpotent.
\end{proposition}
\begin{proof}
Suppose that the proposition is false. Let $G$ be a minimal counter
example. By Lemma \ref{cyclic}, we can assume that $P$ is
non-cyclic.

{\bf Claim $1.$} $O_{2'}(G)=1.$ Assume that $O_{2'}(G)\neq 1.$
Passing to $\bar{G}=G/O_{2'}(G),$ we see that $\bar{G}$ satisfies
the hypothesis of the proposition by Lemma \ref{quot}$(b),$ so that
by inductive hypothesis, $\bar{G}$ is $2$-nilpotent and hence $G$ is
$2$-nilpotent.

{\bf Claim $2.$} If $L\unlhd G$ and $L\neq G,$ then $L\leq O_2(G).$
Assume that $L$ is a proper normal subgroup of $G$ which is not a
$2$-group. As $L\unlhd G,$ $PL$ is a subgroup of $G.$ Assume that
$PL\neq G.$ By Lemma \ref{quot}$(a)$ and the inductive hypothesis,
$PL$ is $2$-nilpotent. Let $Q=O_{2'}(PL).$ Then $1\neq Q\leq L\unlhd
G$ and since $Q$ is characteristic in $L,$ we have $Q\unlhd G$ and
hence $Q\leq O_{2'}(G)=1$ by Claim $2,$ which is a contradiction.
Thus $G=PL.$ Let $U=P\cap L.$ Then $U\in Syl_2(L).$ Suppose that $U$
is not maximal in $P.$ Let $P_1$ be a maximal subgroup of $P$ that
contains $U.$ By comparing the order, we see that $P_1L$ is a proper
subgroup of $PL=G.$ Then by Lemma \ref{max}, $2<|D|<|P_1|$ and so
$P_1L$ is $2$-nilpotent by induction. Arguing as above, we obtain
$1\neq O_{2'}(P_1L)\leq L\unlhd G$ and hence $O_{2'}(P_1L)\leq
O_{2'}(G)=1.$ This contradiction shows that $U$ is maximal in $P.$
Now by Lemma \ref{max} again, $2<|D|<|U|.$ By induction again, $L$
is $2$-nilpotent which leads to a contradiction as above. This
proves our claim.

{\bf Claim $3.$} $N_G(P)$ is $2$-nilpotent. If $N_G(P)<G,$ then it
is $2$-nilpotent by induction and we are done. Thus assume that
$N_G(P)=G.$ Then $P\unlhd G$ and hence every subgroup of $P$ of
order $|D|$ is both subnormal and strongly closed in $G$ so that
they are normal in $G$ by Lemma \ref{quot}$(c).$ By Schur-Zassenhaus
Theorem, there exists a subgroup $A$ of odd order such that $G=PA.$
Since every subgroup of $P$ of order $|D|$ with $2<|D|<|P|$ is
$A$-invariant, by Lemma \ref{fixact}, $A$ centralizes $P$ and hence
$G$ is $2$-nilpotent, which contradicts our assumption.

{\bf Claim $4.$} $F^*(G)=O_2(G).$ As $O_{2'}(G)=1,$ we have
$F^*(G)=O_2(G)E(G).$ Assume that $E(G)\neq 1.$ By Claim $2,$ we have
$E(G)=G$ and then by applying that claim again, we see that $G$ must
be a quasisimple group. Let $H\leq P$ be any subgroup of order
$|D|.$ Assume first that $H\not\leq Z(G).$ Then $H$ is not normal in
$G$ so that $\la H^G\ra=G$ and $P\leq N_G(H)<G.$ By induction
$N_G(H)$ is $2$-nilpotent so that $N_G(H)/C_G(H)$ is a $2$-group. By
Lemma \ref{Syl}, $H\in Syl_2(G),$ which is a contradiction as
$|H|<|P|.$ Thus $H\leq Z(G)$ and since $|D|>2,$ every subgroup of
order $2$ or $4$ is central in $G,$ whence the result follows from
Lemma \ref{min}.

{\bf The final contradiction.} We first show that $P$ is maximal in
$G.$ Let $L$ be any maximal subgroup of $G$ that contains $P.$ By
induction, $L=PO_{2'}(L).$ Since $O_2(G)\unlhd L,$ we obtain
$[O_2(G),O_{2'}(L)]\leq O_2(G)\cap O_{2'}(L)=1,$ hence
$O_{2'}(L)\leq C_G(O_2(G))\leq O_2(G)$ by Lemma \ref{genfit}. Thus
$O_{2'}(L)=1$ which implies that $P$ is maximal in $G.$ Moreover by
Claim $2,$ $O_2(G)$ is a maximal normal subgroup of $G,$ and then
$\bar{G}=G/O_2(G)$ is a simple group with a nilpotent maximal
subgroup $P/O_2(G).$ Assume that $\bar{G}$ is non-solvable. Then  by
Lemma \ref{maxnil}, $\bar{G}\cong L_2(q),$ where $q$ is a prime of
the form $2^m\pm 1\geq 17.$ Let $\bar{M}$ be the maximal subgroup of
$L_2(q)$ which is isomorphic to the dihedral group $D_{2s},$ where
$s>1$ is odd. Let $M,K$ and $A$ be the full inverse images of
$\bar{M},$ the $2$-Sylow subgroup and the cyclic subgroup of order
$s$ of $\bar{M}$ in $G.$ By Schur-Zassenhauss Theorem, $A=O_2(G)T,$
where $|T|=s.$ Also $O_2(G)\leq K\in Syl_2(M)$ and $M=KT,$ where
$A\unlhd M.$ We next show that $|D|\leq |O_2(G)|.$ Assume false.
Then $|O_2(G)|<|D|.$ Now if $|O_2(G)|<|D|/2$ then $\bar{G}$
satisfies the hypothesis of Proposition \ref{fixstrong} with
$|\bar{D}|=|D|/|O_2(G)|,$ and hence $\bar{G}$ is $2$-nilpotent,
contradicts the simplicity of $G.$ Thus we can assume that
$|O_2(G)|=|D|/2.$ Let $H\leq P$ be such that $O_2(G)\leq H$ and
$|\bar{H}|=|H/O_2(G)|=2.$ In this case, $P\leq N_G(H)<G$ and so
$N_G(H)=P$ as $P$ is maximal in $G.$ By Lemma \ref{quot}$(b),$ we
have $N_{\bar{G}}(\bar{H})=\bar{P}.$ Thus $1\neq \bar{H}$ is
strongly closed in $\bar{G}$ and $N_{\bar{G}}(\bar{H})$ is a
$2$-group. By Lemma \ref{Syl}, $\bar{H}=\bar{P}\in Syl_2(G)$ and so
by Lemma \ref{cyclic}, $\bar{G}$ is $2$-nilpotent. This
contradiction shows that $|D|\leq |O_2(G)|.$ Therefore $2<|D|\leq
|O_2(G)|<|K|,$ where $K\in Syl_2(M).$ By induction again,  $M=KT$ is
$2$-nilpotent and thus $O_{2'}(M)=T\unlhd M.$ Hence $T\leq
C_G(O_2(G))\leq O_2(G)$ and then $T=1,$ which contradicts the fact
that $|T|=s>1.$ We conclude that $\bar{G}$ is solvable.  Thus
$\bar{G}$ must be a cyclic subgroup of prime order. Clearly
$|\bar{G}|>2$ otherwise, $G$ is a $2$-group. Let $r=|\bar{G}|$ and
$R\in Syl_r(G).$ Then $G=O_2(G)R$ and $r>2,$ which implies that
$P=O_2(G)\unlhd G,$ and hence $G=N_G(P)$ is $2$-nilpotent by Claim
$3.$ The proof is now complete.
\end{proof}

\noindent
{\bf Proof of Theorem \ref{th1}.}
If $P$ is cyclic or $|D|>2$ or $|D|=2$ but $|P|>2|D|=4$ then the
theorem follows from Proposition \ref{fixstrong}. Thus we can assume
that $P$ is non-cyclic, $|D|=2$ and $|P|=4.$ It follows that every maximal
subgroup of $P$ is strongly closed in $G,$ hence $G$ is $2$-nilpotent by Lemma \ref{max}.  The proof is now complete.

\noindent {\bf Proof of Theorem \ref{th3}.} By Theorem \ref{th2},
$G$ possesses a Sylow tower of supersolvable type. Let $p$ be the
largest prime divisor of $|G|.$ If $p=2,$ then $G$ must be a
$2$-group and hence it is supersolvable. Assume that $p>2.$ The
proof now proceeds as in that of Theorem $3.5$ in \cite{Guo}.

\noindent {\bf Proof of Theorem \ref{th4}.} By Lemma \ref{quot} and
Theorem \ref{th3}, $E$ is supersolvable. Let $p$ be the largest
prime divisor of $|E|.$ If $p>2,$ then the result follows as in
Theorem $3.6$ in \cite{Guo}. Hence we can assume that $p=2$ and so
$E$ is a $2$-group. As $G$ is supersolvable whenever $G$ is a
$2$-group, we also assume that $G$ is not a $2$-group. Since $G/E$
is supersolvable, it has a Sylow tower of supersolvable type and so
$G/E$ is $2$-nilpotent. Let $K/E$ be the normal $2'$-complement of
$G/E.$ By Schur-Zassenhauss Theorem, $K=EA$ where $A$ is of odd
order. Let $E\leq P\in Syl_2(G).$ Then $G=AP,$  where $AE\unlhd G.$
As $|A|$ is odd, $E\in Syl_2(AE)$ and $AE$ satisfies the hypothesis
of Theorem \ref{th1} so that $AE$ is $2$-nilpotent. Hence
$A=O_{2'}(AE)\unlhd AE\unlhd G,$ and so $A\unlhd G.$ We have
$G/A\cong P$ is supersolvable and by hypothesis, $G/E$ is also
supersolvable. Since the class of supersolvable groups is a
saturated formation, we have $G/(A\cap E)\cong G$ is supersolvalbe.
This completes the proof.

\subsection*{Acknowledgment} The author is grateful to the referee
for his or her comments. The author is also grateful to  Prof. Chris
Parker for pointing out reference \cite{Bau} and simplifying the
proof of Lemma \ref{fixact}.

\end{document}